\documentclass[11pt,english,oneside]{amsart}
\usepackage[T1]{fontenc}
\usepackage[latin9]{inputenc}
\usepackage{geometry}
\usepackage{amssymb}
\usepackage{color}

\usepackage{graphicx,color}
\usepackage{amsmath, amssymb, graphics}

\usepackage{graphicx}

\makeatletter
\numberwithin{equation}{section} 
\numberwithin{figure}{section} 
  \@ifundefined{theoremstyle}{\usepackage{amsthm}}{}
  \theoremstyle{plain}
  \newtheorem{thm}{Theorem}[section]
  \theoremstyle{plain}
  
  \theoremstyle{plain}
  \newtheorem{prop}[thm]{Proposition}
  \theoremstyle{Remark}
  \newtheorem{rem}[thm]{Remark}
  \theoremstyle{remark}
  
  \theoremstyle{plain}
  \newtheorem{lem}[thm]{Lemma}
  \newtheorem{mydef}{Definition}

\usepackage{geometry}

\def\bfR#1{{\bf R}^#1}

\def\com#1{ \hbox{#1}}

\def\e{\hbox{\rm e}}

\smallskip
\def\<{{\langle }}
\def\>{{\rangle }}

\makeatother

\usepackage{babel}

\def\bfR#1{{\bf R}^#1}

\def\com#1{ \quad\hbox{#1}\quad}

\def\e{\hbox{\rm e}}

\smallskip
\def\<{{\langle }}
\def\>{{\rangle }}

\makeatother

\begin{document}

\title[cmc tori in $S^3$]{Rotational surfaces in $S^3$ with constant mean curvature.}

\author{ Oscar M. Perdomo}

\date{\today}

\curraddr{O. Perdomo\\
Department of Mathematics\\
Central Connecticut State University\\
New Britain, CT 06050 USA\\
e-mail: perdomoosm@ccsu.edu
}



\maketitle

\begin{abstract}

Recently Ben Andrews and Haizhong Li \cite{AL} showed that every embedded cmc torus in the three dimensional sphere is axially symmetric. There is a two-parametric family of axially symmetric cmc surfaces; more precisely, for every real number $H$ and every $C\ge2(H+\sqrt{1+H^2})$ there is an axially symmetry surface $\Sigma_{H,C}$ with mean curvature $H$. In \cite{P}, Perdomo showed that for every $H$ between $\cot(\frac{\pi}{m})$ and
$\frac{m^2-2}{2\sqrt{m^2-1}}$ there exists an embedded axially symmetric example with non constant principal curvatures that is invariant under the ciclic group $Z_m$. In \cite{AL}, Andrews and Li, showed that these examples are the only  non-isoparametric embedded examples in the family when $H\ge0$. In this paper we study those examples in the family with $H<0$. We prove that there are no embedded examples in the family when $H<0$ and we also prove that for every integer $m>2$ there is a properly immersed example in this family that contains a great circle and is invariant under the ciclic group $Z_m$. We will say that these examples contain the axis of symmetry. Finally we show that every non-isoparametric surface $\Sigma_{H,C}$ is either properly immersed invariant under the ciclic group $Z_m$ for some integer $m>1$ or it is dense in the region bounded by two isoparametric tori if the surface $\Sigma_{H,C}$ does not contain the axis of symmetry or it is dense in the region bounded by a totally umbilical surface if the surface $\Sigma_{H,C}$ contains the axis of symmetry.
\end{abstract}

\section{Introduction}

We say that a surface $\Sigma$ in the three dimensional unit sphere $S^3$ is axially symmetric with respect to the geodesic $\gamma(s)=(\cos s,\sin s, 0,0)$, if  the surface has the form
\begin{eqnarray}\label{the immersions}
\phi(s,t)=(\sqrt{1-|\alpha(t)|^2}\, \cos(s),\sqrt{1-|\alpha(t)|^2}\, \sin(s),\alpha(t))
\end{eqnarray}

where $\alpha: { \bf R}\to \bfR{2}$  is a regular curve contained in the unit circle. The curve $\alpha(t)$ is called the profile curve of the surface $\phi$ and we will say that the surface 
contains the axis of symmetry if the curve $\alpha$ passes through the origin. When the profile curve is a straight line, the surface is a totally umbilical sphere. For  sake of simplicity explaining the result in this paper we will omit totally umbilical spheres from the family of axially symmetric surfaces. When $\alpha$ is a circle centered at the origin, the principal curvatures of the immersion $\phi$ are constant, that is, the surface is isoparametric.
This paper studies axially symmetric surfaces in $S^3$ with constant mean curvature, that is, surfaces with cmc such that up to a rigid motion are of the form (\ref{the immersions}).  From the results in either \cite{W} or \cite{P} we conclude that all the axially symmetric surfaces can be described with two parameter as the set $\{\Sigma_{H,C}: H\in {\bf R},\, C\ge2(H+\sqrt{1+H^2}) \}$. One of the main concerns of this paper is to decide when these surfaces are embedded. With respect to this question, Furuya in 1971 \cite{F}, and independently  Otsuki  in 1972 \cite{O1}  showed that the only axially symmetric minimal hypersurfaces in the $n$ dimensional unit sphere are the Clifford tori- the only isoparametric minimal examples with two principal curvatures. In 1986, Ripol \cite{R} showed that there is a non-isoparametric embedded axially symmetric surface with constant mean curvature $H$,  for any $H$ different from $0$ and $\pm \frac{1}{\sqrt{3}}$. In 1990, Leite and Brito showed the existence of infinitely many embedded non-isoparametric axially symmetric hypersurfaces in $S^{n+1}$ with constant mean curvature. In 2010, Perdomo \cite{P} showed that for every integer $m\ge2$ and any  $H$ between $\cot(\frac{\pi}{m})$ and $\frac{(m^2-2)\sqrt{n-1}}{n\sqrt{m^2-1}}$, there is an non-isoparametric axially symmetric embedded hypersurface in $S^{n+1}$ whose profile curve is invariant under the cyclic group $Z_m$. For axially symmetric surfaces, using complex variables, Li and Andrews, showed that the examples found by Perdomo in \cite{P} are the only embedded axially surfaces in the familly $\Sigma_{H,C}$ when $H\ge0$.  In this paper we will show that none of the examples with $H<0$ are embedded. We will also show that for every $m\ge 3$ there is an axially symmetric surface that contains the axis of symmetric whose profile curve is invariant under the cyclic group $Z_m$. Here are some pictures of  these profile curves.

\begin{figure}[h]\label{special examples}
\centerline{\includegraphics[width=2.5cm,height=2.5cm]{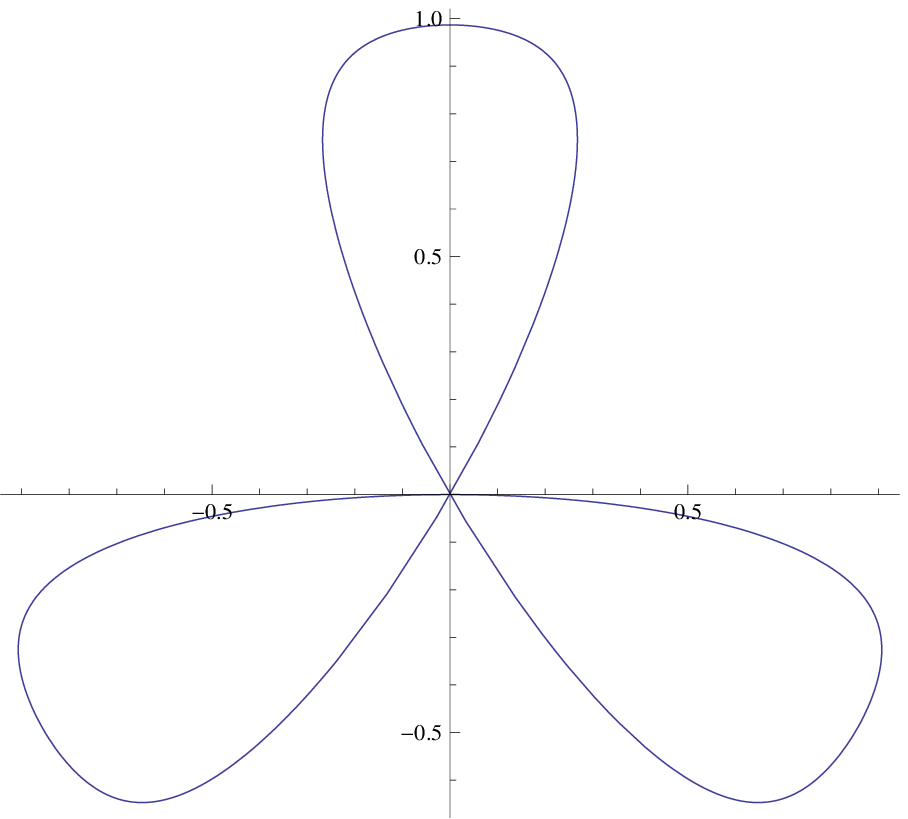}\includegraphics[width=2.5cm,height=2.5cm]{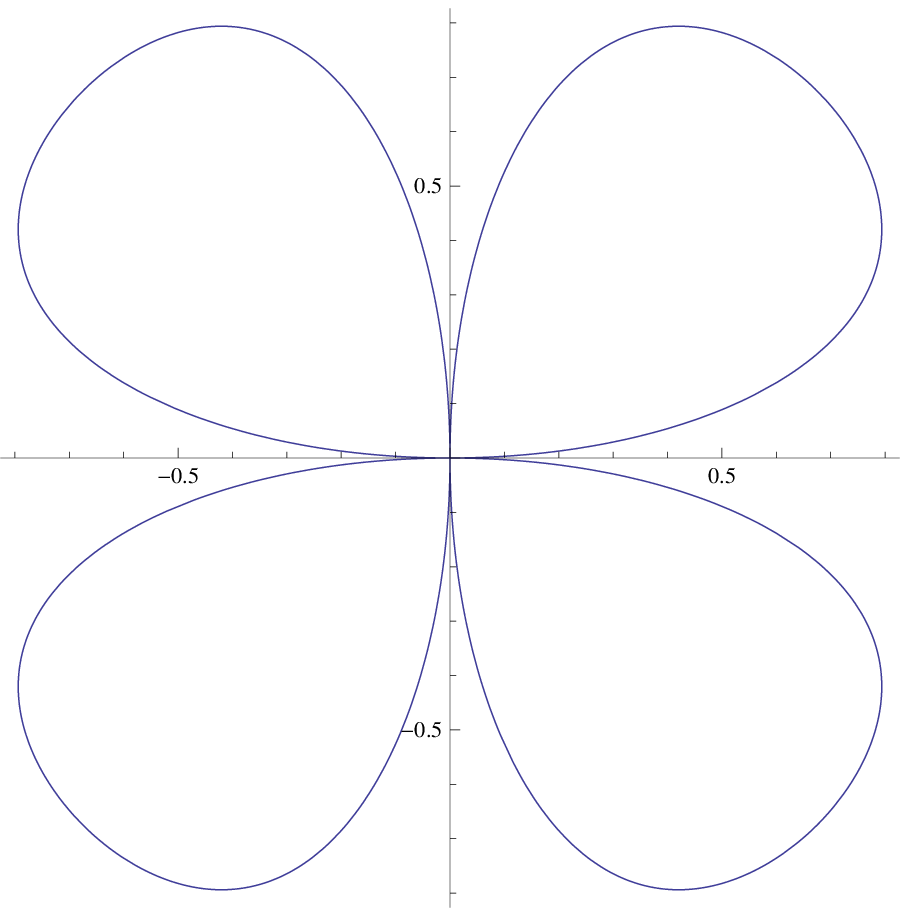}\includegraphics[width=2.5cm,height=2.5cm]{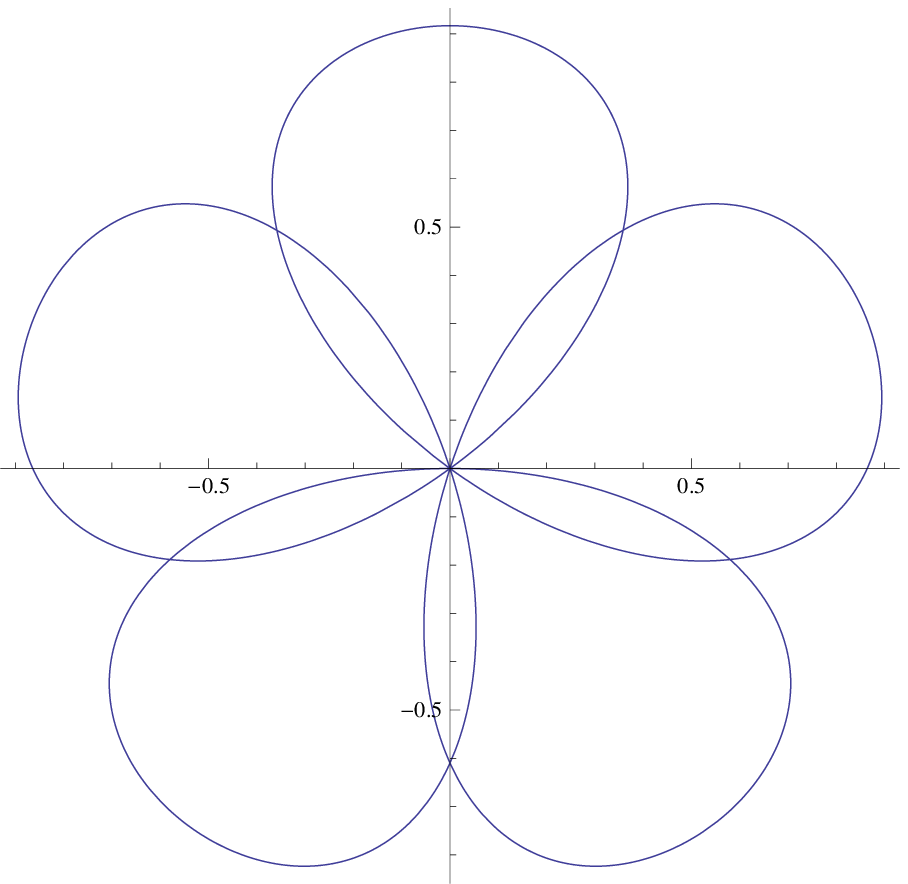}\includegraphics[width=2.5cm,height=2.5cm]{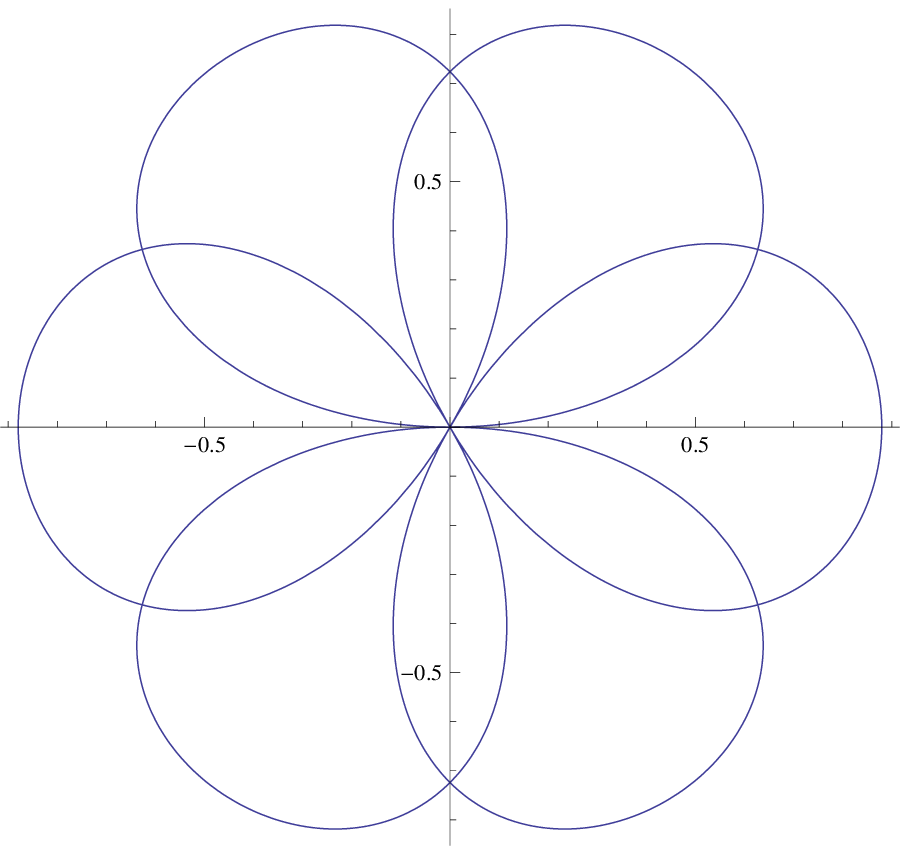}\includegraphics[width=2.5cm,height=2.5cm]{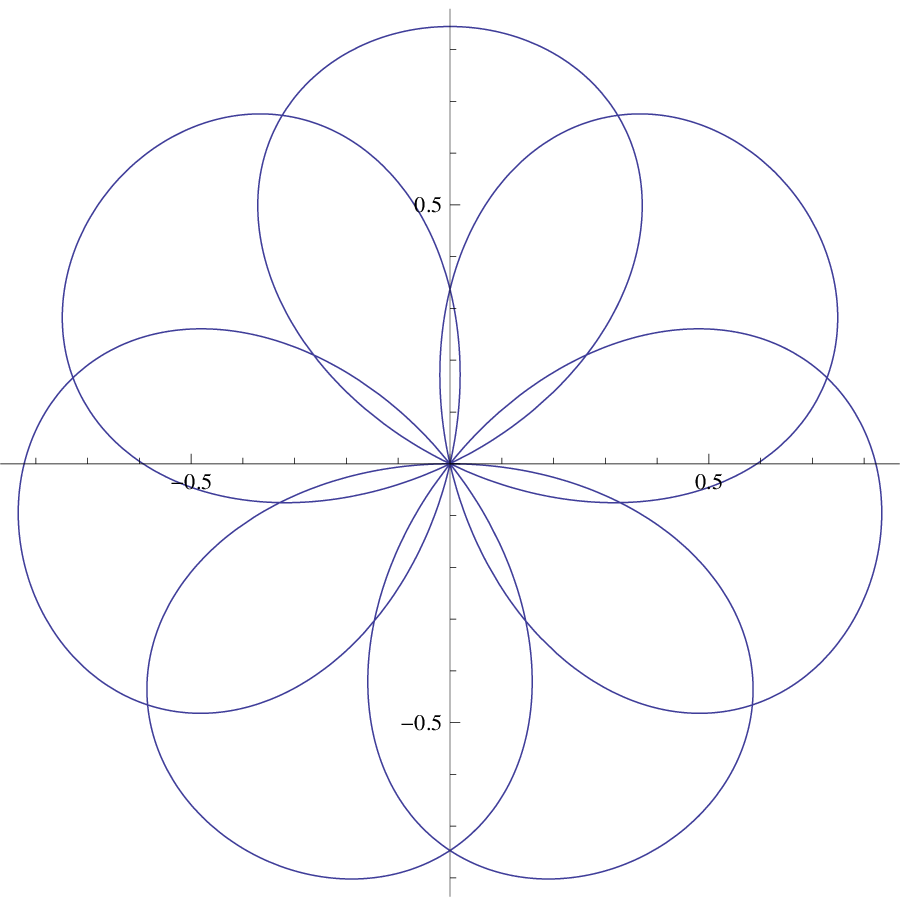}\includegraphics[width=2.5cm,height=2.5cm]{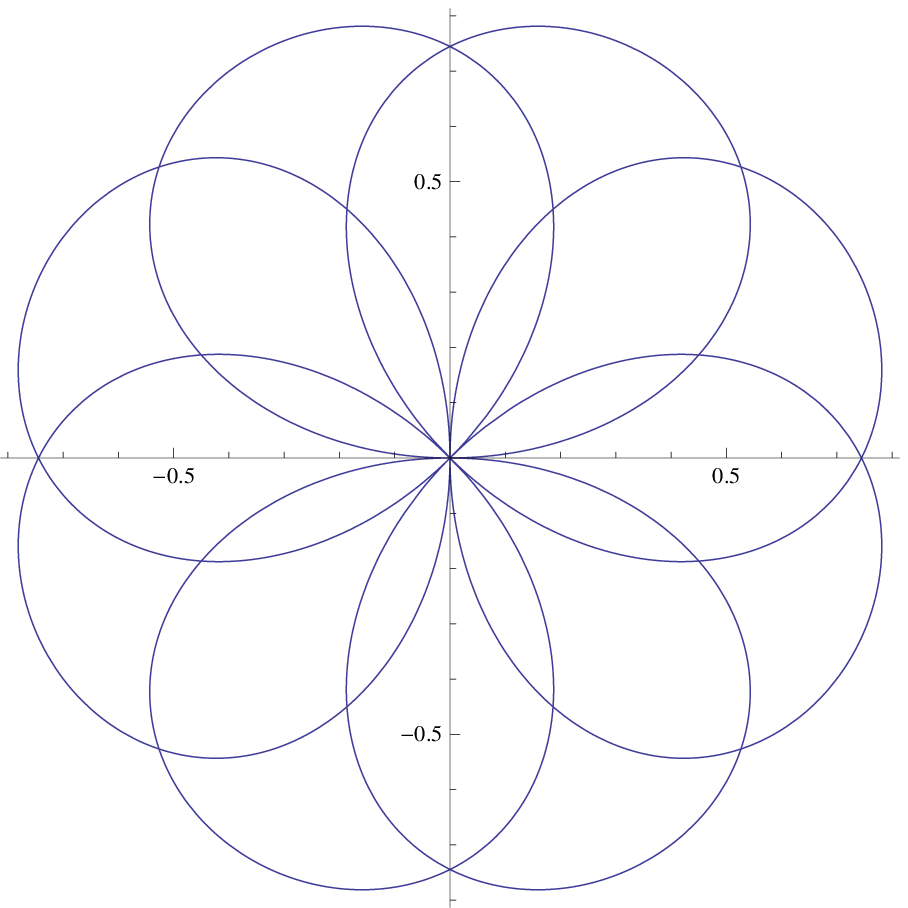}}
\caption{Profile curve of some cmc axially symmetric surface that contains the axis of symmetry.}
\end{figure}

\section{Preliminaries}

It is not difficult to show that for the immersion (\ref{the immersions}), the vector fields $\frac{\partial \phi}{\partial s}$ and $\frac{\partial \phi}{\partial t}$ define principal directions associated with the principal curvatures $\mu$ and $\lambda$  respectively and that for any $s$, the curve $t\to \phi(s,t)$ is a geodesic. The curve $\alpha$ can be re-parametrized  so that, for any $s$, the curves $\gamma(t)=\phi(s,t)$ is parametrized by arc-length.  It is known that cmc axially symmetry surfaces have not umbilical points. Besides, assuming that $\frac{\partial \phi}{\partial t}$  has length 1, we will assume that $\lambda(t)-\mu(t)$ is positive. In 2006, Wei, \cite{Wei} showed that if $g(t)=(\frac{\lambda(t)-\mu(t)}{2})^{-\frac{1}{2}}$ then,


\begin{eqnarray}\label{the ode}
g^\prime(t)^2+g(t)^2\, (1+(H+g(t)^{-2})^2)=C
\end{eqnarray}

This differential equation creates a one to one correspondence between axially symmetric surfaces with constant mean curvature $H$ and positive  solutions of the differential equation. It is not difficult to see that a positive solution exists if and only if $C\ge 2(H+\sqrt{1+H^2})$. When $C=2(H+\sqrt{1+H^2})$, the solution is constant, more precisely $g=(1+h^2)^{-\frac{1}{4}}$. In this case the curve $\alpha$ reduces to a circle centered at the origin with radius $\frac{1}{\sqrt{2+2 H (H-\sqrt{1+H^2})}}$. These solutions are known as Clifford surfaces. A direct computation shows that when $C>2(H=\sqrt{1+H^2})$

\begin{eqnarray}\label{g}
g(t)=\sqrt{  \frac{C-2H}{2+2H^2} +  \frac{\sqrt{C^2-4CH-4}}{2+2H^2}\, \sin(2\sqrt{1+H^2}\, t)   }
\end{eqnarray}

solves the differential equation (\ref{the ode}). We have that $g(t)$ is a periodic function with period $T=\frac{\pi}{\sqrt{1+H^2}}$, that reaches the maximum value $\sqrt{\frac{C-2H+\sqrt{C^2-4CH-4}}{2+2H^2}}$ when $t=\frac{T}{4}$ and the minimum value $\sqrt{\frac{C-2H-\sqrt{C^2-4CH-4}}{2+2H^2}}$ when $t=\frac{3 T}{4}$. Figure 2.1 
shows the graph of the function $g$.

The following lemma gives us explicit immersions, the proof is a direct computation and also can be found in (\cite{P1}, Theorem 2.4)


\begin{lem}\label{explicit formula}
If $C> 2(H+\sqrt{1+H^2})$ and $C\ne -\frac{1}{H}$, then the immersion $\phi$ in (\ref{the immersions}) with 

$$\alpha(t)= \sqrt{\frac{C-g(t)^2}{C}}\, (\cos(\theta(t)),  \sin(\theta(t))) \com{with} \theta(t)=\int_{\frac{T}{4}}^{\frac{T}{4}+t}\frac{\sqrt{C}\, g(\tau)(H+g(\tau)^{-2})}{C-g(\tau)^2}\, d\tau$$

has constant mean curvature $H$. The function $g$ is given by (\ref{g}). We will denote this surface as $\Sigma_{H,C}$. 

\end{lem}


\begin{figure}[h]\label{graph of g}
\centerline{\includegraphics[width=7.5cm,height=4cm]{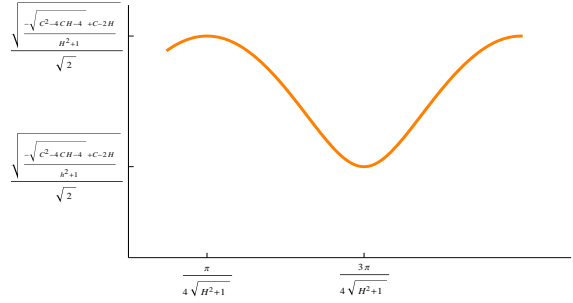}}
\caption{Graph of a non constant positive solution of (\ref{the ode}).}
\end{figure}

\begin{rem}\label{rem2.2}
The function $\theta$ defined in Lemma \ref{explicit formula} is not periodic but, for any $t\in [0,T]$ and any integer $q$, it satisfies that $\theta(t+qT)=\theta(t)+q\theta(T)$ where $T=\frac{\pi}{\sqrt{1+H^2}}$ is the period of the function $g$. Therefore, if 

$$K(H,C)= \theta(T)=\int_{\frac{T}{4}}^{\frac{5T}{4}}\frac{\sqrt{C}\, g(\tau)(H+g(\tau)^{-2})}{C-g(\tau)^2}\, d\tau$$

Then, the immersion in Lemma \ref{explicit formula} is invariant under the subgroup of $O(4)$ given by

$$\left\{\left( \begin{array}{c c c c} 1 & 0 & 0 & 0 \\
                                     0 & 1 & 0 & 0\\  
                                     0 &0 & \cos(q \, K(H,C)) & \sin(q\, K(H,C)) \\
                                     0 &0 & -\sin(q\, K(H,C)) & \cos(q\, K(H,C))  \end{array}\right)  \, : \, q\in Z \right\}$$

This group is finite if and only if, $\frac{K(H,C)}{\pi}$ is rational.
 
\end{rem}

When $H\ge0$ the function $\theta(t)$ is increasing, therefore, if $K(H,C)=\frac{2\pi}{m}$ for some integer $m\ge 1$ and some constant $H$ and $C$, then it is easy to see that the profile curve of $\Sigma_{H,C}$ is a simple curve and therefore $\Sigma_{H,C}$ is embedded. Solutions of the equation $K(H,C)=\frac{2\pi}{m}$ were found in \cite{P} and a proof that these solutions were the only ones when $H\ge0$ was given in \cite{AL} using the following lemma.


\begin{lem}\label{Andrews Li Lemma} {\bf (Andrews-Li \cite{AL})} If $a=\frac{1}{C}$, then 

$$T(H,a)=K(H,\frac{1}{a})=\int_{x_1}^{x_2} \frac{H u+a}{(1-u) \sqrt{u} \sqrt{1+H^2} \sqrt{(u-x_1) (x_2-u)}} \, du$$

where,
$$x_1= \frac{C-2 H-\sqrt{C^2 - 4 C H - 4}}{2 \left(1+H^2\right) C}\com{and} x_2 = \frac{C - 2 H + \sqrt{C^2 - 4 C H - 4}}{2 \left(1+H^2\right) C}$$

Moreover if $H\ge 0$, then  $\frac{\partial T}{\partial a}>0$.

\end{lem}

The variation of the angle function $\theta$  over a period of the function $g$  for the profile curve of minimal axially symmetric hypersurfaces in $S^{n+1}$ was extensively studied by Otsuki. For a paper that reviews the story of this problem we refer to \cite{O2}.  Otsuki was the first one who provided explicit examples of non-isoparametric axially symmetric minimal hypersurfaces in $S^{n+1}$ \cite{O}. 

The main goal of this paper is to decide if, for  any $H<0$, the surface $\Sigma_{H,C}$ is embedded. We will need to deal with three issues. The first one is that even though the family of surfaces $\Sigma_{H,C}$ varies continuously with the parameters $H$ and $C$, the immersions given in Lemma \ref{explicit formula} are not well defined when $C=-\frac{1}{H}$. The second issue is that when $H$ is negative the derivative of the function $\theta$ is positive everywhere for some values of $H$ and $C$ and it changes sign for other values of $H$ and $C$. We need to analyze when each case happens and analyze both cases. The third issue is that the function $K(H,C)$ is not continuous when $C=-\frac{1}{H}$, it has a jump discontinuity. We need to analyze the jump of this discontinuity.


\section{The immersions}

In this section we provide a formula that parametrizes all the immersions $\Sigma_{H,C}$. As a consequence we obtain the continuity of this family of immersions with respect to the parameters $H,C$. It turns out that the reason why the formula given in Lemma \ref{explicit formula} fails when $C=-\frac{1}{H}$ is because it essentially uses polar coordinates for the profile curve and we can check that when $C=-\frac{1}{H}$ the profile curve passes trough the origin. A similar formula for all cmc hypersurfaces in $S^{n+1}$ with axially symmetry symmetry was provided by Wu in  \cite{W}.


\begin{lem}\label{exp for 2}
If $C> 2(H+\sqrt{1+H^2})$, then the immersion $\phi$ in (\ref{the immersions}) with profile curve

$$\beta= \frac{1}{ \sqrt{C {g^\prime}\, ^2+Cg^2}}\, \left(-\sqrt{C} g^\prime \cos(\theta)- (Hg^2+1) \sin(\theta), \sqrt{C} g^\prime \sin(\theta) - (Hg^2+1)\cos(\theta) \right)$$

with

$$ \theta(t)=\int_{\frac{T}{4}}^{\frac{T}{4}+t}\frac{\sqrt{C}\, g(\tau)(H-g(\tau)^{-2})}{{g^\prime(\tau)}\, ^2+g(\tau)^2}\, d\tau$$

has constant mean curvature $H$.

\end{lem}

\begin{proof}
The proof is a direct computation. Also we can show that for any $C \ne -\frac{1}{H}$ with $H<0$, the curve $\alpha$ defined in Lemma \ref{explicit formula} agrees with the curve $\beta$ defined in this lemma. We can see this by directly checking that

\begin{eqnarray}\label{radius}
|\alpha(t)|^2=|\beta(t)|^2=\frac{C+2 H+2 C H^2-\sqrt{-4+C^2-4 C H} \sin(2 \sqrt{1+H^2}\, t)}{2 C+2 C H^2}
\end{eqnarray}
and also checking that if $a(t)$ denotes the angle from the positive direction of the $x$-axis to the position vector $\beta(t)$, then $a^\prime(t)=\frac{\sqrt{C}\, g(t)(H+g(t)^{-2})}{C-g(t)^2}$, which agrees with the derivative of the angle from the  positive direction of the $x$-axis to the position vector $\alpha(t)$.

\end{proof}

\begin{mydef} \label{sigmas}
For any $C>2(H+\sqrt{1+H^2})$, we will refer to the surface in $S^3$ given by the immersion in Lemma \ref{exp for 2}  as $\Sigma_{H,C}$.

\end{mydef}

\section{The function $K(H,C)$}

In this section, using the technique introduced in \cite{AL} we will study the function $K(H,C)$ as a function of $C$ for any real number $H$. We have the following proposition.


\begin{prop} \label{K} For any real number $H$ and any $ C>2(H+\sqrt{1+H^2}) $ different from $-\frac{1}{H}$, let us define $K(C,H)$ as in Lemma (\ref{Andrews Li Lemma}), that is,

$$K(H,C)=\int_{x_1}^{x_2} \frac{H u+C^{-1}}{(1-u) \sqrt{u} \sqrt{1+H^2} \sqrt{(u-x_1) (x_2-u)}} \, du$$

where,

$$x_1= \frac{C-2 H-\sqrt{C^2 - 4 C H - 4}}{2 \left(1+H^2\right) C}\com{and} x_2 = \frac{C - 2 H + \sqrt{C^2 - 4 C H - 4}}{2 \left(1+H^2\right) C}$$

The following statements are true:

\begin{itemize}

\item
The function $C \to K(H,C)$ is decreasing, differentiable and continuous at any value $C\ne-\frac{1}{H}$

\item

$$\lim_{C\to 2(H+\sqrt{1+H^2})^+} K(H,C) = \pi \sqrt{2-\frac{4H}{\sqrt{4+4H^2} }}$$

\item

if $H>0$ then,

$$\lim_{C\to \infty} K(H,C)=2 \, \hbox{\rm arccot}(H) \com{with} 0 < \hbox{\rm arccot}(H)<\frac{\pi}{2} $$

\item

if $H<0$ then,

$$\lim_{C\to \infty} K(H,C)=2 \, \hbox{\rm arccot}(H) \com{with} -\frac{\pi}{2} < \hbox{\rm arccot}(H)<0 $$
\item

if $H<0$ then,

\begin{eqnarray*}
\lim_{C\to -\frac{1}{H}^-} K(H,C)& =& \int_0^1\frac{-H}{\sqrt{v(1-v)(H^2+v)}} \, dv+\pi\\
\end{eqnarray*}
and
\begin{eqnarray*}
\lim_{C\to -\frac{1}{H}^+} K(H,C)& =& \int_0^1\frac{-H}{\sqrt{v(1-v)(H^2+v)}} \, dv-\pi\\
\end{eqnarray*}

\item

The function $b(H)=\int_0^1\frac{-H}{\sqrt{v(v-1)(H^2+v)}} \, dv$ defined in the interval $(-\infty,0)$ is strictly decreasing. Moreover

$$\lim_{H\to - \infty} b(H)=\pi \com{and} \lim_{H\to 0} b(H) = 0$$

\end{itemize}

\end{prop}

\begin{proof}

The proof of the first part of this proposition follows the proof of Lemma \ref{Andrews Li Lemma}  presented in \cite{AL} and we will present the proof here in order to make clear the behavior of the function $K(H,C)$ when $(H,C)$ is near the hyperbola $C=-\frac{1}{H}$. Let $U=\mathbb{C}- \{x+i y\in \mathbb{C}: x\ge0 \, \hbox{and}\, y=0\}$. In the set $U$, let us define the complex function $\hbox{sr}$ as

$$\hbox{sr}(z)=\sqrt{|z|}\, \e^\frac{\theta}{2}\com{where} z=|z|\e^{i\theta}\com{with} 0<\theta<2\pi$$

Clearly the function $\hbox{sr}$ is an analytic function that satisfies $\hbox{sr}^2(z)=z$. A direct computation shows that  $0<x_1<x_2\le1$ and, $x_2=1$ if and only if $H<0$ and $C= -\frac{1}{H}$. Let  $l_1=\{x+iy:y=0,x_1<x<x_2\}$ and $l_2=\{x+iy:y=0, x<0\}$.
Since the mobius transformation $T(z)=\frac{z-x_1}{x_2-z}$ sends the segment $l_1=\{x+iy:y=0,x_1<x<x_2\}$ to the set of positive real numbers, the function $\hbox{sr}(T(z))$ is well defined for all $z\notin l_1$. We also have that the function $\hbox{sr}(-z)$ is well defined for all $z\notin l_2$. Let $\Omega$ be the complement of the set $l_1\cup l_2\cup\{0,x_1,x_2\}$ and let $f:\Omega\to \mathbb{C}$ be the holomorphic function

$$f(z)=-\, i\, \sqrt{1+H^2}\,  \hbox{sr}(-z)\, (x_2-z)\, \hbox{sr}(\frac{z-x_1}{x_2-z})$$

Let $\gamma_1$, $\gamma_2$ and $\gamma_3$ be the curves given in Figure 4.1.
 Using the fact that for every $x_1<x<x_2$ and and $\epsilon\ne0$ we have that

$$\lim_{\epsilon\to 0^+} \hbox{sr}\left(\frac{(x+i \epsilon)-x_1}{x_2-(x+i\epsilon)}\right)=\sqrt{\frac{x-x_1}{x_2-x}}\com{and} 
\lim_{\epsilon\to 0^-} \hbox{sr}\left(\frac{(x+i \epsilon)-x_1}{x_2-(x+i\epsilon)}\right)=-\sqrt{\frac{x-x_1}{x_2-x}}$$

and

$$ \lim_{\epsilon\to 0} \hbox{sr} (-(x+i\epsilon))=i\, \sqrt{x}  $$ 

we can prove, as pointed out in \cite{AL} that,

$$K(H,C)= \int_{x_1}^{x_2}   \frac{H u + C^{-1}}{(1-u) \sqrt{u} \sqrt{1+H^2} \sqrt{(u-x_1) (x_2-u)}} \, du =-\frac{1}{2}\int_{\gamma_1} \frac{Hz+C^{-1}}{(1-z)f(z)}\, dz$$


\begin{figure}[h]\label{curves}
\centerline{\includegraphics[width=7cm,height=7cm]{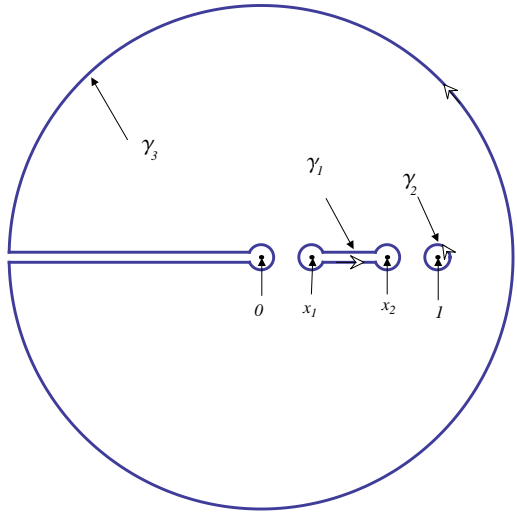}}
\caption{The closed curves $\gamma_1$, $\gamma_2$ and $\gamma_3$}
\end{figure}

Since the function $\frac{H+C^{-1}}{(1-z)f(z)}$ is holomorphic, we have that 

$$\int_{\gamma_3}\frac{Hz+C^{-1}}{(1-z)f(z)}\, dz+\int_{-\gamma_1}\frac{Hz+C^{-1}}{(1-z)f(z)}\, dz+\int_{-\gamma_2}\frac{Hz+C^{-1}}{(1-z)f(z)}\, dz=0$$

Therefore,

\begin{eqnarray}\label{K}
K(H,C)=\left\{\begin{array}{l c r}-\pi- \frac{1}{2} \, \int_{\gamma_3} \frac{Hz+C^{-1}}{(1-z)f(z)}\, dz &\hbox{if} & C>-\frac{1}{H}\\
            & & \\
              \pi- \frac{1}{2} \, \int_{\gamma_3} \frac{Hz+C^{-1}}{(1-z)f(z)}\, dz &\hbox{if} & C<-\frac{1}{H} \end{array} \right.
\end{eqnarray}

A direct computation shows that 

$$\frac{\partial\left( \frac{  H z+C^{-1}}{(1-z) f(z)} \right) }{\partial C}=-\frac{z^2}{C^2f(z)^3}$$

Therefore, we obtain that

$$\frac{\partial K(H,C)}{\partial C}=\frac{1}{2 C^2} \int_{\gamma_3}\frac{z^2}{f(z)^3}$$

Taking the limit when the bigger radius in the curve $\gamma_3$ goes to infinity and the small radius in the curve $\gamma_3$ goes to $0$ we obtain that

$$\frac{\partial K(H,C)}{\partial C}=-\frac{1}{C^2}\, \int_{-\infty}^0\frac{x^2}{\sqrt{(1+H^2)(-x)(x_1-x)(x_2-x))}}\, dx < 0$$

A direct computation shows that if $H<0$, then

 $$\lim_{C\to -\frac{1}{H}} x_2=1\com{and} \lim_{C\to -\frac{1}{H}} x_1=\frac{H^2}{1+H^2}$$
 
 We also have that, on the curve $\gamma_3$, 
 
 $$ \frac{  H z+C^{-1}}{(1-z) f(z)}\longrightarrow G(z)= \frac{H}{i(1-z) \hbox{sr}(-z) \, \hbox{sr}(\frac{(1+H^2)z-H^2}{1-z})} \com{as} C\to -\frac{1}{H}$$ 
 
 The function $G$ is holomorphic in the complement of the set 
 
 $$\{x+i y:\,y=0 \,\hbox{and,} \,  \hbox{either}\, x\le0 \quad \hbox{or}\quad \frac{H^2}{1+H^2}\le x\le 1\}$$

\begin{figure}[h]\label{g4}
\centerline{\includegraphics[width=7cm,height=7cm]{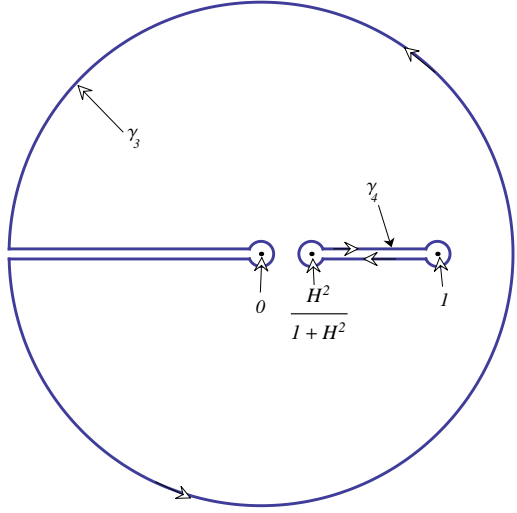}}
\caption{The closed curves  $\gamma_3$ and $\gamma_4$}
\end{figure}

If we define $\gamma_4$ like in Figure 4.2, 
we have that $\int_{\gamma_3}g(z)dz+\int_{\gamma_4}g(z)dz=0$, and taking the limit when the radius of the small circle in the curve $\gamma_4$ goes to zero we obtain that

$$-\frac{1}{2}\int_{\gamma_3} G(z)\, dz = \frac{1}{2}\int_{\gamma_4} G(z)\, dz=\int_{\frac{H^2}{1+H^2}}^{1} \frac{-H}{\sqrt{u(1-u)((1+H^2)u-H^2)}}\, du= \int_0^1\frac{-H}{\sqrt{v(1-v)(H^2+v)}} \, dv$$

The last equation follows by doing the substitution $v=(1+H^2)u-H^2$. Using the expression for $K(H,C)$ given in \ref{K} we obtain the expressions in the Proposition  for the limit of $K(H,C)$ when $C$ goes to $-\frac{1}{H}^+$ and $-\frac{1}{H}^-$. Using the fact that $b(H)=-2 i H EllipticK(-H^2)+2 H EllipticK(1+H^2)$, we obtain that $\lim_{H\to0^-}b(H)=0$ and $\lim_{H\to-\infty}b(H)=\pi$. 

The limits for $K(H,C)$ when $C$ goes to $2(H+\sqrt{1+H^2})$ and when $C$ goes to $\infty$ were computed in \cite{P}.

\end{proof}

Figure 4.3 
shows the limit values of the function $K(H,C)$ as well as all possible values of this function for every $H$. It also shows the discontinuity of this function when $C$ approaches $-\frac{1}{H}$.
\begin{figure}[h]\label{limits}
\centerline{\includegraphics[width=9cm,height=7cm]{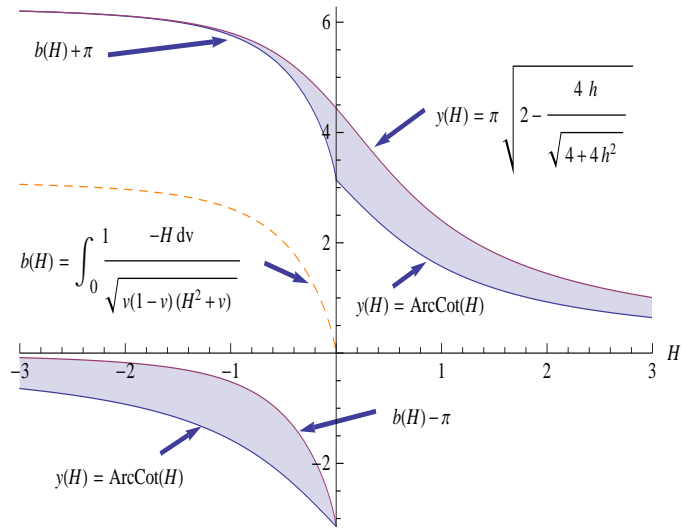}}
\caption{The shaded region shows the possible values of $K(H,C)$ for every $H$}
\end{figure}

\section{Fundamental Piece}

The profile curve of cmc rotational surface in $S^3$ can be built as the union of rotations of a single piece that we will call the fundamental piece.

\begin{mydef}

The fundamental piece of the surface $\Sigma_{H,C}$ is defined as the profile curve $\beta(t)$ given in Lemma \ref{exp for 2} restricted to the interval $[-\frac{\pi }{4 \sqrt{1+H^2}},\frac{3 \pi }{4 \sqrt{1+h^2}}]$

\end{mydef}

Here are some easy-to-check properties of the fundamental piece of the surfaces $\Sigma_{H,C}$.

\begin{prop}\label{pfp} Let $\gamma:[-\frac{\pi }{4 \sqrt{1+H^2}},\frac{3 \pi }{4 \sqrt{1+h^2}}]\to \bfR{2}$ be the fundamental piece of the surface $\Sigma_{H,C}$. The following properties hold true:

\begin{itemize}
\item The distance from points in the fundamental piece to the origin satisfies the following inequality:
$$ | \gamma(\frac{\pi }{4 \sqrt{1+H^2}})|\le |\gamma(t)|\le |\gamma(\frac{3\pi }{4 \sqrt{1+H^2}})|=|\gamma(-\frac{\pi }{4 \sqrt{1+H^2}})|$$
\item
If $C\ne-\frac{1}{H}$, then, the angle between $\gamma(-\frac{\pi }{4 \sqrt{1+H^2}})$ and $\gamma(\frac{3\pi }{4 \sqrt{1+H^2}})$ is given by $K(H,C)$. If $C=-\frac{1}{H}$, then, the angle
 between $\gamma(-\frac{\pi }{4 \sqrt{1+H^2}})$ and $\gamma(\frac{3\pi }{4 \sqrt{1+H^2}})$ is given by $b(H,C)=\int_0^1\frac{-H}{\sqrt{v(v-1)(H^2+v)}} \, dv$
\item
The profile curve of $\Sigma_{H,C}$ is the union of rotations of the fundamental piece.
\end{itemize}
\end{prop}

\begin{proof}

The first item follows from the inequality (\ref{radius}). The second item follows from Remark 2.2 
and the continuity in therm of $H$ and $C$ of the immersions given in Lemma \ref{exp for 2}. The last item again follows from Remark 2.2
\end{proof}
\begin{figure}[h]\label{fp}
\centerline{\includegraphics[width=6cm,height=6cm]{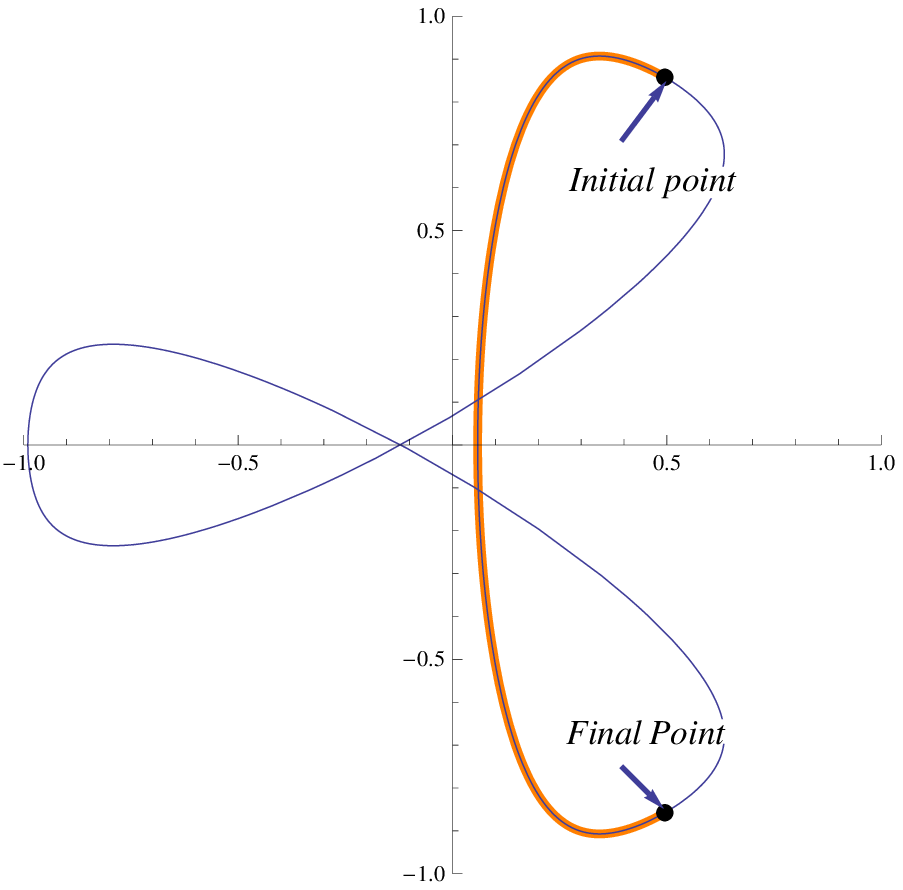}\includegraphics[width=6cm,height=6cm]{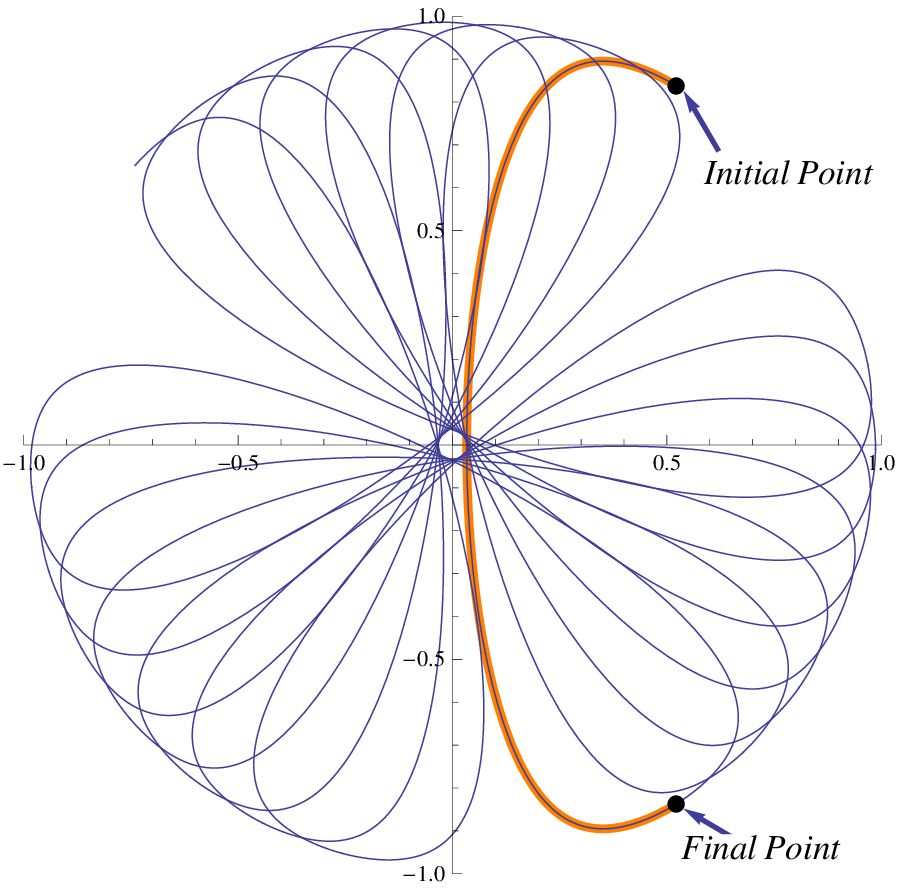}}
\caption{The first picture shows the profile curve and its fundamental piece when $H=-0.2$ and $C=7.10621080709656$. The second picture shows part of the profile curve when $H=-0.2$ and $C=6$}
\end{figure}

The next theorem shows a properly immersed vs dense duality property for the immersions $\Sigma_{H,C}$.

\begin{thm} If $C=-\frac{1}{H}$ with $H<0$, then the immersion $\Sigma_{H,C}$ is either properly immersed or it is dense in the region 

$$\{  (x,y,z,w)\in S^3:  z^2+w^2\le \frac{1}{1+H^2}\}$$

Moreover, if $C\ne -\frac{1}{H}$, then the immersion $\Sigma_{H,C}$ is either properly immerse or it is dense in the region 

$$R_{H,C}= \{ (x,y,z,w)\in S^3: m_{H,C}\le z^2+w^2\le M_{H,C}\}$$

where,

$$m_{H,C}= \frac{C+2 H+2 C H^2-\sqrt{-4+C^2-4 C H}}{2 C+2 C H^2},\quad  M_{H,C}= \frac{C+2 H+2 C H^2+\sqrt{-4+C^2-4 C H} }{2 C+2 C H^2}$$

\end{thm}

\begin{proof}
When $C\ne -\frac{1}{H}$ we have that if $\frac{K(H,C)}{2 \pi}=\frac{m}{k}$ with $k$ and $m$ relatively prime integers and $k>0$, then, it is easy to see that the profile curve is the union of $k$ copies of rotations of the fundamental piece. Therefore the immersion is proper. If $\frac{K(H,C)}{2 \pi}$ is an irrational number, for any $r\in [ \sqrt{m_{H,C}}, \sqrt{M_{H,C}}]$, using the first item of Proposition \ref{pfp} and the intermediate value theorem, we have that as $t$ goes from $0$ to $\infty$, the profile curve $\beta(t)$ of the surface $\Sigma_{H,C}$, hits the circle $C_r$ centered at the origin with radius $r$ at points that differ by a fixed angle $K(H,C)$. The union of all these points in the circle $C_r$ is dense in the $C_r$. The problem of proving this last statement reduces to that of showing that for any irrational number $\tau$, the set $\{n \tau-[n \tau]:n\in \mathbb{Z}\}$ is dense in the interval $[0,1]$, which is a known fact. We therefore have that the profile curve is dense in the annulus, $\{  (z,w)\in \bfR{2}: m_{H,C} \le z^2+w^2\le M_{H,C}\}$. Looking at the formula for rotational immersion (Formula \ref{the immersions}), we conclude that $\Sigma_{H,C}$ is dense in $R_{H,C}$. The proof of the case $C=-\frac{1}{H}$ is similar.

\end{proof}

\section{Embedded examples}

Let us start this section by considering the rotational immersions obtained when $C=-\frac{1}{H}$ with $H<0$. As pointed out before, for this case the immersions given in Lemma \ref{explicit formula} are not well defined and therefore the immersions given in Lemma \ref{exp for 2} are needed.

\begin{thm}
None of the cmc rotational immersions in $S^3$ that satisfies $C=-\frac{1}{H}$ is embedded. Moreover, the profile curves of all these examples contain the origin and therefore these immersions contain  a circle of 
radius 1. Also,  for any $m>2$, there exists an $H<0$ such that the profile curve of the rotational immersion $\Sigma_{H,C}$, given $C=-\frac{1}{H}$, is invariant under the cyclic group $Z_m$.  Figure 1.1 in the introduction shows some of these immersions.


\end{thm}

\begin{proof}

Let us start pointing out that, since $2(H+\sqrt{1+H^2})=\frac{2}{\sqrt{1+H^2}-H}$, then, we have that for every $H<0$,  $2(H+\sqrt{1+H^2})<-\frac{1}{H}$. Therefore, for any $H<0$ there is a rotational surface in $S^3$ such that $C=-\frac{1}{H}$. If $C=-\frac{1}{H}$, then, the values of the function $g$ moves from  $g(\frac{3 \pi}{4\sqrt{1+H^2}})=\sqrt{\frac{1-H}{1+H^2}}$ to  $g(\frac{\pi}{4\sqrt{1+H^2}})=\sqrt{-\frac{1}{H}}$. By the continuity with respect to the variables $H$ and $C$ of the immersion given in \ref{exp for 2} and Proposition \ref{K}, we have that the angle between the initial and final point is given by $b(H)=\int_0^1\frac{-H}{\sqrt{v(v-1)(H^2+v)}} \, dv$. We also know from Proposition \ref{K} that $b(H)$ takes every value in the open interval $(0,\pi)$. Therefore for any $m>2$ we can find a $H_m<0$ such that $b(H_m)=\frac{2\pi}{m}$. Since for this value of $H_m$ the profile curve is the union $m$ fundamental pieces, then, the immersion $\Sigma_{H_m,C}$ with $C=-\frac{1}{H_m}$ is properly immersed. These examples are invariant under the cyclic group $Z_m$. Since $0<b(H)<\pi$, then for all $H<0$ and $C=-\frac{1}{H}$ we have that the profile curve of $\Sigma_{H,C}$ does not close up when $t$ moves from $0$ to $2 \frac{\pi}{\sqrt{1+H^2}}$. A direct computation shows that anytime $H<0$ and $C=-\frac{1}{H}$ then the profile curve $\beta(t)$ satisfies that  $\beta(\frac{\pi}{4 \sqrt{1+H^2}})=\beta(\frac{5\pi}{4 \sqrt{1+H^2}})=(0,0)$. Therefore, none of the immersion with  $C=-\frac{1}{H}$ is embedded and all of them have the great circle $(\cos(s),\sin(s),0,0)$. This finishes the proof of the theorem.

\end{proof}

For the other immersions $\Sigma_{H,C}$ with $H<0$ we have,

\begin{thm}\label{emb}
None of the immersion $\Sigma_{H,C}$, with $H<0$ and $C\ne-\frac{1}{H}$ is embedded
\end{thm}

For these immersions we can use the formula given in Lemma \ref{explicit formula} because $C\ne-\frac{1}{H}$. We will use the curve $\alpha$ and $\theta$ defined in Lemma \ref{explicit formula}. Recall that the function $g$ has period $T=\frac{\pi}{\sqrt{1+H^2}}$ and  reaches the maximum value $\sqrt{\frac{C-2H+\sqrt{C^2-4CH-4}}{2+2H^2}}$ when $t=\frac{T}{4}$ and the minimum value $\sqrt{\frac{C-2H-\sqrt{C^2-4CH-4}}{2+2H^2}}$ when $t=\frac{3 T}{4}$. A direct computation shows that if $2 (H+\sqrt{1+H^2})<C<-\frac{1}{H}$,
then $H+g(\tau)^{-2}$ is always positive because

$$H+g(\tau)^{-2}\ge H+\left(  \sqrt{\frac{C-2H+\sqrt{C^2-4CH-4}}{2+2H^2}}\right)^{-2}=\frac{1}{2} \left(C-\sqrt{-4+C^2-4 C H}\right)$$

and we can easily check that $C-\sqrt{-4+C^2-4 C H}>0 $ if $2 (H+\sqrt{1+H^2})<C<-\frac{1}{H}$. Since $\theta^\prime(t)=\frac{\sqrt{C}\, g(\tau)(H+g(\tau)^{-2})}{C-g(\tau)^2}$, we conclude that the function $\theta$ is always increasing when $C<-\frac{1}{H}$. By Proposition \ref{K} we have that the angle swept by the fundamental piece, $K(H,C)$, is between
$\sqrt{2}\pi$ and  $2 \pi$ when $C<-\frac{1}{H}$.

\begin{figure}[h]\label{emb12}
\centerline{\includegraphics[width=6cm,height=6cm]{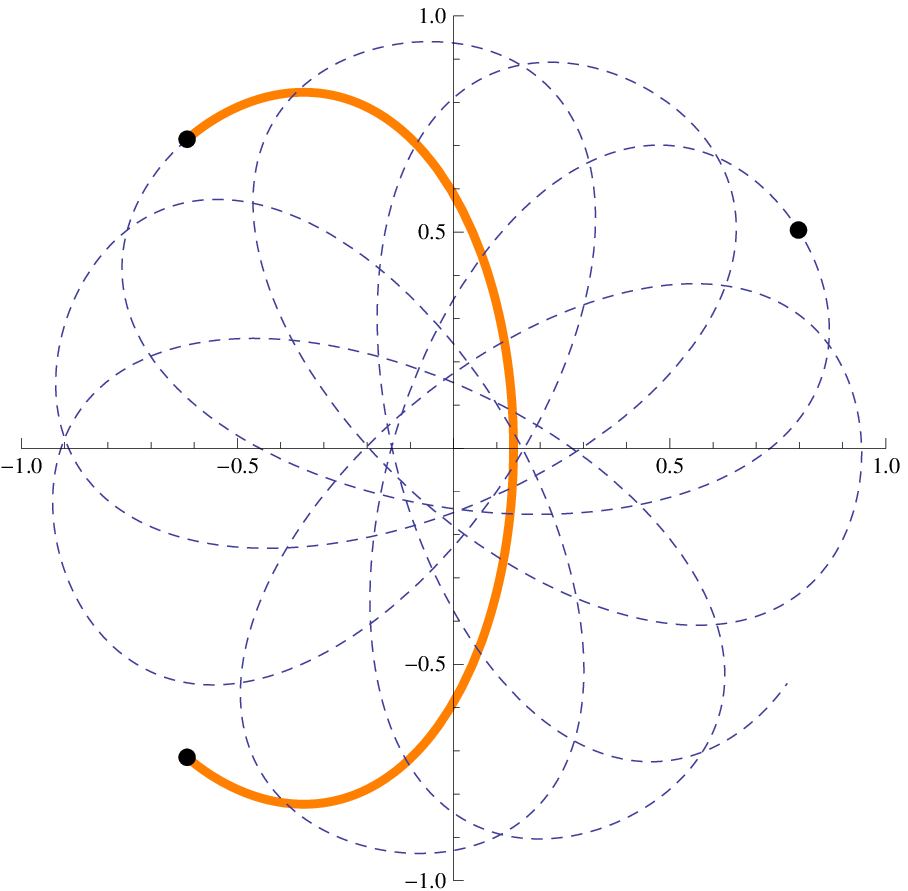}\includegraphics[width=6cm,height=6cm]{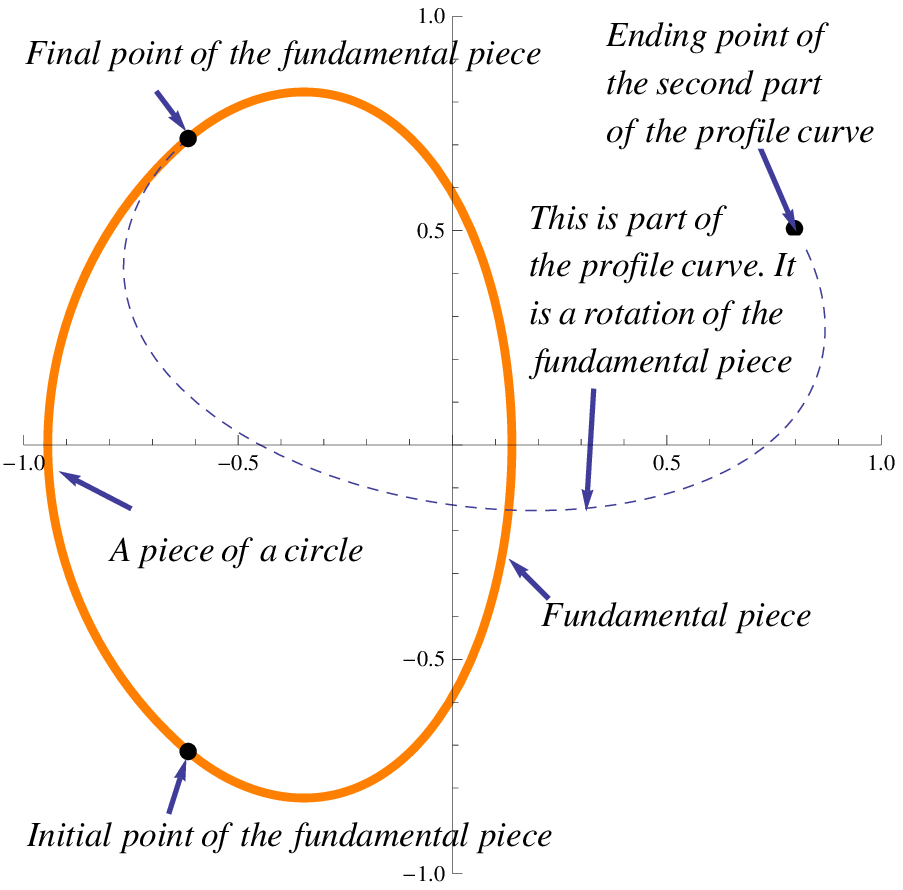}}
\caption{The first picture shows part of the profile curve when $H<0$ and $C<-\frac{1}{H}$. The second picture shows the simple closed curve used in the proof of Theorem \ref{emb} }
\end{figure}

We will prove that the immersion $\Sigma_{H,C}$ is not embedded when $C<-\frac{1}{H}$ by showing that the profile curve self intersects. Since the function $\theta$ is increasing we have that the fundamental piece do not self intercept. On the other hand, since $\sqrt{2}\pi<K(H,C)<2 \pi$, then we have that we can form a simple closed curve making the union of the fundamental piece and a circular arc with a central angle less than $\pi$, $2\pi-K(H,C)$ to be precise. See Figure 6.1. Again, using the fact that $\sqrt{2}\pi<K(H,C)<2 \pi$, we get that the portion of the profile curve defined for values of $t$ between
$\frac{3 \pi}{4 \sqrt{1+H^2}}$ and $\frac{7 \pi}{4 \sqrt{1+H^2}}$ will contains points inside and outside the simple closed curve.  Notice that for all 
$t\in(\frac{3 \pi}{4 \sqrt{1+H^2}},\frac{7 \pi}{4 \sqrt{1+H^2}})$, we have that $|\alpha(t)|<|\alpha(\frac{3 \pi}{4 \sqrt{1+H^2}})|$ and $\alpha(\frac{7 \pi}{4 \sqrt{1+H^2}})$ 
is outside the closed curve because the angle swept by the portion of profile curve consisting of two fundamental pieces is greater than $2 \pi$ and smaller than $4 \pi-(2\pi-K(H,C)$. By the Jordan closed curve theorem we conclude that the profile curve must self intercept. For the rest of this proof we will assume that $C>-\frac{1}{H}$. Since
$\theta^\prime(-\frac{\pi}{4 \sqrt{1+H^2}})$ and $\theta^\prime(\frac{3 \pi}{4 \sqrt{1+H^2}})$ are both positive and $0<K(H,C)<\pi$, we can obtain a closed $C^1$ curve by attaching a circular arc with central angle smaller than $\pi$ -negative $K(H,C)$ to be precise- to the a fundamental piece of the profile curve (see Figure 6.2). We will call this closed curve, the curve $\eta$.

\begin{figure}[h]\label{emb3}
\centerline{\includegraphics[width=6cm,height=6cm]{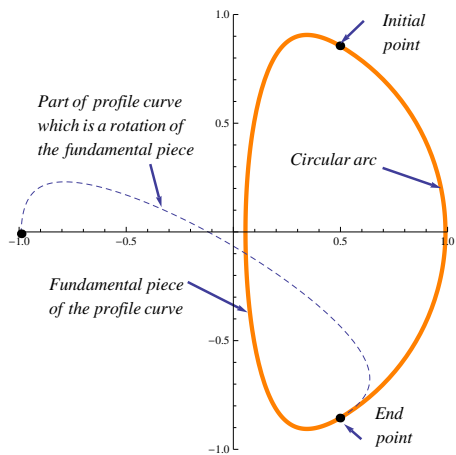}}
\caption{This picture shows a fundamental piece of the profile curve of a cmc rotational surface when $H<0$ and $C>-\frac{1}{H}$. It also shows the closed curve used in the proof of Theorem \ref{emb} }
\end{figure}

Once again, we want to show that the rotationally surface $\Sigma_{H,C}$ is not embedded by showing that the profile curve self intersects. Without loss of generality, we can assume that $\eta$ is a closed simple curve. Since $-\pi<K(H,C)<0$ we obtain that $-2\pi<2K(H,C)<K(H,C)$, therefore, using similar arguments as before we conclude that the portion of the profile curve defined for values of $t$ between
$\frac{3 \pi}{4 \sqrt{1+H^2}}$ and $\frac{7 \pi}{4 \sqrt{1+H^2}}$ will contain points inside and outside the simple closed curve $\eta$. By the Jordan closed curve theorem we conclude that the profile curve must self intercept. This finish the proof of the Theorem.

\begin{rem}
Roger Vogeler, a colleague of the author, pointed out that  another proof of the fact that the surfaces $\Sigma_{H,C}$ when $H<0$ are not embedded can be shown by proving that  the angle swept by the unit tangent vector of the fundamental piece is between $\pi$ and $2\pi$ because it is known that the total variation of the unit tangent vector of a simple closed curve is $2\pi$.
\end{rem}

\end{document}